\documentclass[12pt,leqno,twoside]{amsart}

\usepackage[latin1]{inputenc}
\usepackage[T1]{fontenc}
\usepackage[colorlinks=true, pdfstartview=FitV, linkcolor=blue, citecolor=blue, urlcolor=blue]{hyperref}
\usepackage{amstext,amsmath,amscd, bezier,indentfirst,amsthm,amsgen,enumerate, geometry}
\usepackage[all,knot,cmtip,arc,import,poly]{xy}
\usepackage{amsfonts,color, soul}  %\st{} for overstriking
\usepackage{amssymb}
\usepackage{latexsym}
\usepackage{epsfig} 
\usepackage{graphicx}
\usepackage{srcltx}
\usepackage{enumitem}
\usepackage{mathtools}
\usepackage{tikz-cd}
\usepackage[all]{xy}

\usepackage{tikz} % pacote grafico
\usetikzlibrary{through} % circulo passando por ponto, por exemplo.
\usetikzlibrary{patterns} % preenchimentos
\usetikzlibrary{intersections} % interseccao entre caminhos
\usetikzlibrary{matrix} % matriz no tikz
\usetikzlibrary{cd} % diagrama comutativa

\newtheorem{theorem}{Theorem}[section]
\newtheorem{corollary}[theorem]{Corollary}
\newtheorem{lemma}[theorem]{Lemma}
\newtheorem{proposition}[theorem]{Proposition}
\theoremstyle{definition}
\newtheorem{definition}[theorem]{Definition}
\theoremstyle{remark}
\newtheorem{remark}[theorem]{\sc Remark}
\newtheorem{example}[theorem]{\sc Example}
\newtheorem{question}[theorem]{\sc Question}

\newcommand{\Sing}{{\rm{Sing\hspace{2pt}}}}

\newcommand{\rank}{{\rm{rank\hspace{2pt}}}}

\newcommand{\Disc}{{\rm{Disc\hspace{2pt}}}}

\renewcommand{\d}{{\rm{d}}}

\renewcommand{\t}{{\rm{T}}}

\newcommand{\m}{\setminus}

\newcommand{\bR}{{\mathbb R}}

\newcommand{\bN}{{\mathbb N}}

\newcommand{\R}{\mathbb{R}}

\begin{document}
	
	\title{Tameness conditions and the Milnor fibrations for composite singularities}
	
	\author{\sc R. N. Ara\'ujo dos Santos}
	\address{\sc Raimundo N. Ara\'ujo dos Santos: ICMC, Universidade de S\~ao Paulo, Av. Trabalhador S\~ao-Carlense, 400 - CP Box 668, 13560-970 S\~ao Carlos, S\~ao Paulo, Brazil.}
	\email{rnonato@icmc.usp.br}
	
	\author{\sc D. Dreibelbis}
	\address{\sc Daniel Dreibelbis: University of North Florida, 1 UNF Dr., Jacksonville, FL, 32258, USA}
	\email{ddreibel@unf.edu}
	
	\author{\sc M. F. Ribeiro}
	\address{\sc M. F. Ribeiro: UFES, 	Universidade Federal do Esp\'irito Santo,  Av. Fernando Ferrari, 514 - CEP 29.075-910 Vit\'oria, Espirito Santo,  Brazil}
	\email{maico.ribeiro@ufes.br}
	\email{mfsr83@gmail.com}

	\author{\sc I. D. Santamar\'ia Guar\'in}
	\address{\sc I. D. Santamar\'ia Guar\'in: ICMC, Universidade de S\~ao Paulo, Av. Trabalhador S\~ao-Carlense, 400 - CP Box 668, 13560-970 S\~ao Carlos, S\~ao Paulo, Brazil.}
	\email{ivansantamariag89@usp.br}

	
	
	\begin{abstract} {
			
			In this paper, we introduce a new regularity condition that characterizes the tameness of a composite singularity $H=G\circ F$ in a sharp way. Our approach provides a natural tool that links the topology of the Milnor tube fibrations through the Milnor fibers of the respective components of the map germs $F$, $G$ and $H = G\circ F$. We also study the invariance of tameness by $\mathcal{L}$-equivalence, $\mathcal{R}$-equivalence, and hence by $\mathcal{A}$-equivalence, and we give conditions for when two component map germs of the composite singularity $H=G\circ F$ being tame implies the third one is tame. As an application, we show how to relate the Euler characteristics of the Milnor fibers of $F,G$ and $H$ to each other.}
		%Let $F:(\bR^M,0)\to (\bR^N,0), M\geq N\geq 2,$ be an analytic map germ with $\Disc F=\{0\}$ and satisfying a tame condition, or equivalently the Milnor condition $(b)$ at origin in the sense of D. Massey in \cite{DA}. It is well known that the Milnor's tube fibration exists and the Milnor fiber is the local smoothing of the special fiber $V_{F}:=F^{-1}(0)$ which is unique up to diffeomorphism. One may consider now another analytic map germ $G:(\bR^N,0)\to (\bR^K,0), N\geq K\geq 2,$ with $\Disc G=\{0\}$ and hence the composition of map germ $H=G\circ F$ also have $\Disc H=\{0\}.$ One might ask then whether or nor the tameness condition from the map germs $F$ and $G$ implies that for the composition $H$. It would allow several interesting applications, among them, it would provide a powerful tool to relate the geometrical and the topological data of the Milnor fiber (or, the smoothing) of $H$ from those of the Milnor fibers of $F$ and $G.$ 
		%In this paper we develop a study showing how to relate the singular sets and the Milnor sets of $F,G$ and $H$ in order to introduce a condition (Condition \eqref{eqHtame}) that 
		%characterize the tameness of the composition map germ $H=G\circ F.$ Our condition is sharp and it is satisfied by all the previous cases. In particular, it provides an answer to the aforementioned question. 
		%Moreover, we also introduce a sufficient condition under which the tameness of $F$ and $G$ assures the tameness of the composition $H$ as well.
		
	\end{abstract}
	
	\maketitle
	
	\section{Introduction}
	
	Let $F:(\bR^M,0)\to (\bR^N,0), M\geq N\geq 2,$ be an analytic map germ with isolated critical value, that is, $F(\Sing F):=\Disc F$ is just the origin and satisfies a tame condition, or equivalently the Milnor condition $(b)$ at the origin in the sense of D. Massey in \cite{DA}. It is well known that a Milnor tube fibration exists and the Milnor fibers are a local smoothing of the special fiber $V_{F}:=F^{-1}(0)$, which is unique up to diffeomorphism.  If we have another analytic map germ $G:(\bR^N,0)\to (\bR^K,0), N\geq K\geq 2$, %with $\Disc G=\{0\}$ 
	, we can consider the germ of the composition map $H=G\circ F$, and then ask the following question:
	
	\begin{question} \label{q1}
		Given $F$, $G$, and $H = G \circ F$, under what conditions on $F$ and $G$, one can ensure that $\Disc H = \{0\}$ and that $H$ satisfies a tameness condition?
	\end{question}
	
	In \cite{CT23}, the authors showed that under the restrictive condition of isolated singularity for $G$, i.e., $\Sing G= \{0\}$, the tameness for the composition $H$ is guaranteed, provided that $\Disc F = \{0\}$ and $ F $ satisfies the tame condition. Furthermore, it was proved in \cite{RADG} that any harmonic morphism is tame and has an isolated critical value at the origin. See also \cite{AR} for a construction of harmonic complex functions.  Furthermore, the composition of harmonic morphisms is a harmonic morphism. So, for such a special class of functions, Question \eqref{q1} is true. 
	
	\vspace{0.3cm}
	
	In this paper we develop a study without imposing any condition under the dimension of the singular set of $G$, as long as it is still included in $V_G$, showing how to relate the singular sets and the Milnor sets of $F,G$ and $H$ to introduce a condition (Condition \eqref{eqHtame}) that characterizes the tameness of the composition map germ $H$ and shows that it may not depend on the tameness of $G$. Our condition is sharp and is satisfied by all previously known cases. In particular, it provides an answer to the aforementioned question. Moreover, we also introduce a sufficient condition under which the tameness of $F$ and $G$ assures the tameness of the composition $H$ as well, generalizing the main result of the paper \cite{CT23}. 
	
	\vspace{0.3cm}
	
	The paper is organized as follows. In Section 2, we recall the main results about the existence of a local tube fibration. In Section 3, we relate the singular sets and the discriminants of $ F $, $ G $, and $ H $. We find an important relationship between the Milnor sets of $ F $ and $ H $ and show that the Milnor set is invariant by ${\mathcal L}$-equivalence, and it is invariant by $\mathcal{R}$ and $\mathcal{A}$-equivalence if we consider in the source space the pullback of the metric $\rho$ of the target space.  In Section 4, we study the problem of tameness of the composite. We show that, in general, the tameness of $F$ and $G$ are not enough to guarantee the tameness of H, under the less restrictive hypothesis $\dim \Sing G >0$. However, we introduce a sufficient condition for the tameness of $H$, under the hypothesis that $F$ and $G$ are both tame. In the final section, we show that our results have several interesting applications; among them, they provide a powerful tool for relating the geometric and topological data of the Milnor fiber (or smoothing) of $H$ to those of the Milnor fibers of $F$ and $G$.  For example, we show how the Euler characteristics of the Milnor fibers of $F,G$ and $H$ and their Milnor tubes are related to each other.    
	
	\section{Existence of the Milnor's tube fibrations}

	Here we introduce the Milnor tube fibration for a real analytic map germ. The existence of these fibrations has been extensively studied by mathematicians since the 1980s. For detailed information, refer to the following references \cite{Mi,Le,Ma,DA,Ri, ADER} and \cite{ART}, for further
	generalities.
	%For the sake of simplicity we will use the same notations for the germs and its representatives. 

	\vspace{0.2cm}
	
	Let $F:(\bR^M,0) \to (\bR^N,0),$ $M\geq N\geq 2,$ be a non-constant real analytic map germ and let $F: U \to \bR^{N},$ $F(0)=0$ be a representative of the germ, where $U \subseteq \mathbb{R}^{M}$ is an open set and $F(x)=(F_1(x),F_2(x), \ldots, F_N(x)).$ From now on we will denote a germ and its representative by the same notation.
	
	\vspace{0.3cm}
	
	Denote the zero locus $ F^{-1}(0)$ of $F$ by $ V_F$ and the singular set by $\Sing F:=\{x\in U: \rank(d F(x))~\text{is not maximal}\}$, where $\d F(x)$ denotes the Jacobian matrix of $F$ at $x$. The discriminant set of $F$ is given by $\Disc(F):=F(\Sing(F)).$ We say that $F$ has an isolated critical value at origin if $\Disc(F)=\{0\}$, that is, $ \Sing F \subseteq V_F $. 
	
	\vspace{0.2cm}
	
	In this paper, all analytic map germs will be considered non-constant functions. 
	Moreover, as already mentioned above, without loss of generality, all sets and maps will be considered as a germ at the origin, once they are well defined. 
	
	\vspace{0.2cm}
	
	\begin{definition} We say that an analytic map germ $F:(\bR^M,0) \to (\bR^N,0)$ admits a Milnor's tube fibration if there exist $\epsilon_{0}>0$ and $\eta>0$ such that for all $0< \epsilon \leq \epsilon_{0}$, in the closed ball $\bar{B}^M_{\epsilon}$ the restriction map 
		
		\begin{equation}\label{tmil}
		F_{|}:\bar{B}^M_{\epsilon}\cap F^{-1}(\bar{B}_{\eta}^{N}-\{0\})\to \bar{B}_{\eta}^{N}-\{0\}.
		\end{equation}
		
		\noindent is a locally trivial smooth fibration for all $0<\eta \ll \epsilon.$ Moreover, the diffeomorphism type does not depend on the choice of $\epsilon$ and $\eta$.
		
	\end{definition}
	
	%In such a case, one easily sees that such a fibration restricts to a smooth fibration in the open ball 
	
	%\begin{equation}\label{tmil}
	%\dfrac{F}{\|F\|}:B_{\epsilon}^m\cap F^{-1}(S_{\eta}^{n-1})\to S^{n-1}.
	%\end{equation}
	
	To state the existence of the tube fibration we need to introduce some standard notation and definitions. Let $U \subset \mathbb{R}^M$ be an open set, $0\in U$, and let $\rho:U \to \mathbb{R}_{\ge 0}$ be a smooth nonnegative and proper function that is zero only at the origin, for example, the square of the Euclidean distance to the origin. 
	\begin{definition}\label{d:ms}
		Given $F:(\mathbb{R}^M, 0) \to (\mathbb{R}^N,0), M\geq N\geq 1,$ an analytic map germ. The set-germ at the origin
		\[M_\rho(F):=\Sing(F, \rho) \]
		is called the \emph{Milnor set of $F$}, or the set of \textit{$\rho$-nonregular points} of $F$, or even the \textit{ polar set of $F$ relative to $\rho$}.
	\end{definition}
	
	It follows from the definition that $\Sing F \subset M(F)$. Moreover, the complement set $M(F) \m \Sing F$ is the set of points where the non-singular fibers of $F$ and $\rho$ do not intersect in general position, i.e., the points of nontransversality between the fibers of $F$ and $\rho .$ 
	
	\vspace{0.2cm}
	
	For most of this paper, we consider only the Euclidean distance function $\rho_{E}(x)=\|x\|^{2}$, which we denote for short $M(F):= M_{\rho_{E}}(F)$. Under mild adaptations, all the results stated in this paper easily extend to any $\rho$ function as defined above.
	
	\vspace{0.2cm}
	
	\begin{definition}
		We say that a real analytic map germ $F:(\bR^M,0) \to (\bR^N,0)$ is \textit{tame}\footnote{This condition was first used in \cite{Ma} under the name \textit{Milnor condition (b) at the origin}, and later in \cite{ACT} to ensure the existence of a Milnor fibration.} if 
		\begin{equation}\label{eq:main}
		\overline{M(F)\m V_{F}}\cap V_{F} \subseteq \{0\}.
		\end{equation}
		as a set germ.
	\end{definition}
	
	\vspace{.2cm}

	We note that our definition of tameness depends on the choice of $\rho$.  In particular, it is possible for a function to be tame with respect to one choice of $\rho$, but not tame with respect to a different choice (see Remark~\ref{rem:nottame}).  Thus, whenever we claim in this paper that a function is not tame, we are saying that it is not tame with respect to $\rho_{E}(x)=\|x\|^{2}$.
	
	\vspace{0.2cm}
	
	It is well known that the $a_{f}$-Thom regularity along the singular fiber $V_{F}$ implies Condition \eqref{eq:main}, see \cite[Proposition 4.2]{ART}. The converse is not true, and in \cite{PT} and more recently in \cite{R},  several classes of map germs were introduced that satisfy Condition \eqref{eq:main} but are not Thom regular. Another interesting class of map germs satisfying Condition \eqref{eq:main} is the ICIS class of map germs. Let us recall that the ICIS condition for a map $F$ amounts to the condition $\Sing F \cap V_F \subseteq \{0\}$.   Therefore, Condition \eqref{eq:main} is not restricted to map germs with isolated critical value at the origin.   
	
	\vspace{0.2cm}
	
	The following result aids in our understanding of the local behavior of an analytic function defined on an analytic manifold, following the work in \cite{DA}. It is a very classical statement and the proof can be done as an easy application of the Curve Selection Lemma. For more details, see, for instance, \cite{Mi, Loo}. 
	
	\vspace{0.2cm}
	
	Let $Y \subset \bR^M$ be an analytic set of dimension $d$ that contains the origin $0$ and let $ g: \bR^M \to  \bR$ be an analytic function such that $g(0)=0$. Also, let $\Sigma_{Y}$ denote the singular set of $Y$. 
	
	\begin{lemma}\cite[Lemma 3.1]{DA}\label{reg} The critical points of $g_{|Y\m \Sigma_{Y}}$ lie in $\{g=0\}$ in a neighbourhood of the origin.
	\end{lemma}
	\begin{proof}
		It is a direct application of the Curve Selection Lemma, which we may safely leave to the reader.
	\end{proof}

	As an immediate consequence of Lemma~\ref{reg}, we will reinterpret Condition~\eqref{eq:main} in a more compact way. See also \cite{Loo,Mi,DA} and \cite{CT23}. Indeed, decompose $V_F = \Sing F \sqcup \left(V_F \m \Sing F\right)$ into a disjoint union of semi-analytic varieties\footnote{Or, semi-algebraic if $F$ is polynomial.}. If $V_F \m \Sing F$ is not empty, it is a semi-analytic manifold. One may take $g(x)=\rho(x)=\|x\|^2$ and $Y=V_{F}$ and, after applying Lemma \ref{reg}, one can see that the analytic manifold $V_F \m \Sing F$ must be transversal to all level sets $g=\epsilon^{2}$ (spheres of radii $\epsilon$), for all $\epsilon >0$ small enough. Hence, in a small neighborhood of the origin one has $M(F)\cap (V_{F}\m \Sing F)=\emptyset$ . Therefore,
	
	$$\overline{M(F)\m V_F} \cap V_F = \overline{M(F)\m \Sing F} \cap \Sing F,$$ 
	
	\vspace{0.2cm}
	
	The next result ensures the existence of a Milnor tube fibration.
	
	\vspace{0.2cm}
	
	\begin{theorem}[Existence of a Milnor tube fibration, \cite{ART1,Ma}]\label{ttf} Let $F: (\bR^{M},0)\to (\bR^{N},0),$ $M\geq N \geq 2,$ be a real analytic map germ with $\Disc(F)=\{0\}$ and $F$ tame. Then $F$ admits a Milnor tube fibration (\ref{tmil}).
	\end{theorem}

	The fibration below is also called the \it{Milnor tube fibration} and it follows as a direct consequence of the Theorem above.
	
	\begin{corollary}\label{ttf2} Let $F: (\bR^{M},0)\to (\bR^{N},0),$ $M\geq N \geq 2,$ be a real analytic map germ with $\Disc(F)=\{0\}$ and $F$ tame. Then, there exists an $\epsilon_{0}>0$ small enough such that for each $0<\epsilon< \epsilon_{0},$ there exists $\delta$, $0<\delta \ll \epsilon$ such that the restriction map $F_{|}:\bar{B}^{N}_{\epsilon}\cap F^{-1}(S_{\delta}^{N-1})\to S_{\delta}^{N-1}$ is the projection of a smooth locally trivial fibration.  
	\end{corollary}

	\section{The Milnor sets of the composites singularities}
	\label{sec:composite}
	
	Let $F:(\mathbb{R}^M, 0) \to (\mathbb{R}^N,0)$ and $G:(\mathbb{R}^N, 0) \to (\mathbb{R}^K,0)$, $ M\ge N \ge K \ge 2 $ be analytic map germs. Let $H:(\mathbb{R}^M, 0) \to (\mathbb{R}^K,0)$ be the composition map germ $H=G \circ F$.
	
	\vspace{0.2cm}
	
	\subsection{The Singular set of a composition}
	
	It follows that $V_{H}:=H^{-1}(0)=F^{-1}(V_{G})$. By the chain rule, the Jacobian matrix $\d H(x)=\d G(F(x))\cdot \d F(x)$, hence 
	\begin{equation}\label{singco}
	\Sing H \subseteq \Sing F \cup F^{-1}(\Sing G).
	\end{equation}
	It follows that if both $\Disc F =\{0\}$ and $\Disc G =\{0\}$ then $\Disc H =\{0\}.$ 
	
	\vspace{0.2cm}
	
	In general, the inclusion \eqref{singco} may be strict even if $\Disc F =\{0\},$ as can be seen in the next example.
	
	\begin{example}
		Consider $ F(x,y,z,w)=(x,y,z(x^2+y^2+z^2+w^2)) $, $ G(u,v,t)=(u,v)$ and $H=G \circ F$. One has $ \Sing G = \emptyset  \subset V_G$, $ \Sing F = \{0\} \subset \{x=y=z=0\} = V_F $ and $ \Sing H = \emptyset .$ 
	\end{example}
	
	\begin{lemma}\label{l0}
		Let $F:(\mathbb{R}^M, 0) \to (\mathbb{R}^N,0)$ and $G:(\mathbb{R}^N, 0) \to (\mathbb{R}^K,0)$, $M\ge N \ge K \ge 2,$ be analytic map germs such that $\Disc F =\{0\}$ and $\{0\} \subseteq \Sing G.$ Then $ \Sing H = F^{-1}(\Sing G)$.
	\end{lemma}
	\begin{proof}
		If $\Disc F =\{0\}$ and $\{0\} \subseteq \Sing G$ then $ \Sing H \subseteq F^{-1}(\Sing G).$ Conversely, for $ x\in F^{-1}(\Sing G)$, $ y=F(x) \in \Sing G$, then $ \d G(y): \R^{N} \to \R^{K} $ is not onto and, in particular, it is not onto when restricted to the subspace $\d F(x)(\R^M) \subseteq \R^N.$ Thus, $x \in \Sing H$.
	\end{proof}
	
	\begin{remark}\label{cex1} The hypothesis $ \{0\} \subseteq \Sing G $ in Lemma~\ref{l0} cannot be weakened. Indeed, consider $ F(x,y,z,w)=(x,y,z(x^2+y^2+z^2+w^2)) $ and $ G(u,v,t)=(u,t) $. Then $ \Sing G = \emptyset = F^{-1}(\Sing G)$, but $ \Sing H = \{0\}.$
	\end{remark}
	
	From Condition~\eqref{singco}, we know that $\Disc F = \{0\}$ and $\Disc G = \{0\}$ imply $\Disc H = \{0\}$.  If we assume $\{0\} \subseteq \Sing G$, then applying Lemma~\ref{l0} shows that $\Disc F = \{0\}$ and $\Disc H = \{0\}$ imply $\Disc G = \{0\}$.  However, it is not true that $\Disc G = \{0\}$ and $\Disc H = \{0\}$ imply $\Disc F = \{0\}$ (for example, let $G$ be any map with $\Disc G = \{0\}$, and let $F$ be a diffeomorphism).  The following diagram illustrates the connections between the discriminant sets of $F,G$ and $H$. Following the directions of the arrows, each two implies the third. The red cut arrow means that the implication does not hold.  For example: $\Disc G=\{0\}$ and $\Disc H=\{0\}$ does not imply $\Disc F=\{0\}$.

	%{\color{blue} We have shown that $\Disc F = \{0\}$ and $\Disc G = \{0\}$ implies $\Disc H = \{0\}$, that $\Disc F = 0$ and $\Disc H = 0$ implies $\Disc G = 0$, but $\Disc G = 0$ and $\Disc H = 0$ do not guarantee that $\Disc F = 0$.}  The following diagram illustrates the connections between the discriminant sets of $F,G$ and $H$. Following the directions of the arrows, each two implies the third. {\color{red} The cut red arrow means that the implication does not hold.  For example: $\Disc G=\{0\}$ and $\Disc H=\{0\}$ does not imply $\Disc F=\{0\}.$}
	
	\vspace{0.4cm}
	
	%%%%%%%%%%%%%%DIAGRAM 0 %%%%%%%%%%%%%%%%
	{\color{blue}
		\[
		\begin{tikzcd}[row sep=2.5em]
		& \fbox{$\Disc F=\{0\}$} \arrow[dl, shift left]  \arrow[dr,  shift right] \\
		\fbox{$\Disc G=\{0\}$} \arrow[rr, shift left] \arrow[ur, "/"{anchor=center, sloped}, red, shift left ] && \fbox{$\Disc H=\{0\}$} \arrow[ul, "/"{anchor=center, sloped}, red, shift right] \arrow[ll, shift left]
		\end{tikzcd}
		\]
	}
	%%%%%%%%%%%%%%DIAGRAM 1 %%%%%%%%%%%%%%%%
	
	%\begin{center}
	%\begin{tikzcd}[sep=huge]
	%\fbox{\Disc F =\{0\}} \arrow[r,shift left] \arrow[d,shift left] &
	%\fbox{\Disc G =\{0\}}  \arrow[dl,shift left] \\
	%\fbox{\Disc H =\{0\}} \arrow[ur,shift left]
	%\end{tikzcd}
	%\end{center}
	
	%\vspace{0.2cm}
	%%%%%%%%%%%%%DIAGRAM 2%%%%%%%%%%%%%
	
	% \arrow[r, "/" marking]
	
	%\[
	%\begin{tikzcd}[row sep=2.5em]
	%& \fbox{\Disc F=\{0\}} \arrow{dl}{a}  \arrow{dr}{g} \\
	%\fbox{\Disc G=\{0\}} \arrow{rr} && \fbox{\Disc H=\{0\}} 
	%\end{tikzcd}
	%\]
	
	%\vspace{0.2cm}
	%%%%%%%%%%%%%%DIAGRAM 3%%%%%%%%%%%%%%%%
	
	%\begin{center}
	%\begin{tikzcd}
	%\fbox{\Disc F=\{0\}} \arrow[r, red, shift left=1.5ex] \arrow[r]
	%\arrow[dr, blue, shift right=1.5ex] \arrow[dr]
	%& \fbox{\Disc G=\{0\}} \arrow[d, purple, shift left=1.5ex] \arrow[d]\\ & \fbox{\Disc H=\{0\}}
	%\end{tikzcd}
	%\end{center}
	
	\subsection{The Milnor set of the composite maps}
	Considering both $F$ and $G$ with an isolated critical value at the origin, one cannot expect in general that $M(H) \subseteq M(F)$. As an example, let $F(x,y,z,w)=(x,y,z)$ and $G(u,v,t)=(u,v(u^2+v^2+t^2))$. Then $M(F)=\{w=0\}$ and $ M(H)=\{x=y=z=0\} \cup \{z=w=0\} $. 
	
	\vspace{0.2cm}
	
	In this subsection we will state and prove some important results concerning the Milnor set of the composite singularity. Once more we point out that all results in this section holds for an arbitrary $\rho: U\to \bR_{\geq 0}$ smooth nonnegative and proper function.

	\vspace{0.2cm}
	
	\begin{lemma}\label{l1}
		Let $F:(\mathbb{R}^M, 0) \to (\mathbb{R}^N,0)$ with $\Disc F =\{0\},$ $G:(\mathbb{R}^N, 0) \to (\mathbb{R}^K,0)$, $ M\ge N \ge K \ge 2 $, be analytic map germs. Then $$\Sing H \subset M(H) \subseteq M(F) \cup \Sing H. $$ In particular, if $ \{0\} \subseteq \Sing G $ then $ M(H) \m \Sing H \subseteq  M(F) \m V_F$.
	\end{lemma}
	\begin{proof}
		
		Let $x \in M(H)\m \Sing H$ and consider the matrix
		
		\begin{center}
			$A(x)=\left[ \begin{array}{c}
			\d H(x) \\ 
			\d \rho (x)
			\end{array}\right] = \left[ \begin{array}{c}
			\d G(F(x))\cdot \d F(x)  \\ 
			\d \rho (x)
			\end{array}\right]= \left[ \begin{array}{c}
			\d G_{1}(F(x))\cdot \d F(x)  \\ 
			\vdots \\ 
			\d G_{k}(F(x))\cdot \d F(x)   \\ 
			\d \rho (x)
			\end{array}\right] . $
		\end{center}
		
		Since $x \in M(H)\m \Sing H$, $\textrm{rank}(A(x))$ is not maximal. However, the $\textrm{rank}~\d H(x)$ is maximal and for $j=1,\ldots,k $ each line $\d G_{j}(F(x))\cdot \d F(x) $ is a linear combination of the gradients $\d F_{i}(x), i=1,\ldots, n$ with the coefficients of the partial derivatives of $G_{j}.$ Hence $x\in M(F)$. 
		
		\vspace{0.2cm}
		
		Moreover, if $\{0\} \subseteq \Sing G$ then $ V_F:=F^{-1}(0) \subset F^{-1}(\Sing G)=\Sing H$ and hence $M(H) \m \Sing H \subseteq  M(F) \m V_F.$
	\end{proof}
	
	\vspace{0.2cm}
	
	Next result is an immediate consequence of the Lemma \ref{l1} but we state it for completeness.
	
	\vspace{0.2cm}
	
	\begin{proposition}\label{p1} Let  $F:(\mathbb{R}^M, 0) \to (\mathbb{R}^N,0)$ and  $G:(\mathbb{R}^N, 0) \to (\mathbb{R}^K,0)$, $ M\ge N \ge K \ge 2 $,  analytic map germs such that $G$ is a germ of submersion. Then 
		\begin{enumerate}
			\item [(i)] $ M(H) \subseteq M(F) $; 	
			\item [(ii)] If $N=K$, then  $M(H)=M(F).$ 
		\end{enumerate}
	\end{proposition}
	
	\begin{proof} The item $(i)$ is an immediate application of the Lemma \ref{l1}. 
		
		\vspace{0.2cm}
		
		For the item $(ii)$ it is enough to prove that the inclusion $ M(F) \subset M(H)$ holds true. For that, take $x\in M(F)$ and $\lambda(x) \in \bR^N$ such that $ \d \rho(x) = \lambda(x)\cdot \d F(x) $. If $N=K$ then $G$ is a local diffeomorphism in a neighborhood of the origin, and hence $ \d G(F(x)): \R^N \to \R^N $ is an isomorphism for all $x$ close enough to the origin. Thus, one can find a unique $\beta(x) \in \R^N $ such that $ \beta(x)\cdot \d G(F(x)) = \lambda(x) $. Consequently, $ \d \rho(x) = \beta(x) \cdot \d G(F(x))\cdot \d F(x) =\beta(x)\cdot \d(G\circ F)(x)=\beta(x)\cdot \d H(x)$. Therefore, $ x\in M(H) $ and $M(H)=M(F).$
	\end{proof}

	\begin{corollary}
		Let $ F_1,F_2:(\mathbb{R}^M, 0) \to (\mathbb{R}^N,0)$ be analytic map germs that are $C^\infty - \mathcal{L}-$equivalent. Then $M(F_1)=M(F_2)$.  Furthermore, $F_1$ is tame if and only if $F_2$ is tame.
	\end{corollary}
	
	\begin{proof}
		Our hypothesis says there is a diffeomorphism germ $G:(\mathbb{R}^N, 0) \to (\mathbb{R}^N, 0)$ such that $F_1 = G \circ F_2$.  It follows from Proposition \ref{p1} that $M(F_1) = M(F_2)$, and $\Sing F_1 = \Sing F_2$ is immediate.  Thus Condition~\eqref{eq:main} is equivalent for both functions.
	\end{proof}
	
	\vspace{.2cm}
	
	Right-equivalence (and thus $\mathcal{A}$-equivalence) will not preserve $M(F)$. In fact, it will not even preserve the topological structure of $M(F)$.  For example, for the family of maps $F(x, y, z) = (x y, y z(a x^2+y^2+z^2))$ with $a \ge 1$, $M(F)$ transitions from the union of a cone and a plane when $a<9$ to a double cone and a plane when $a>9$, even though all maps in the family are $\mathcal{A}$-equivalent (see Figure~\ref{aequiv}).  Thus we cannot expect tameness with respect to a particular $\rho$ to be preserved.  However, we can say that the map will be tame with respect to some proper function.  
	
	\begin{figure}	
		\includegraphics[width=.9\textwidth]{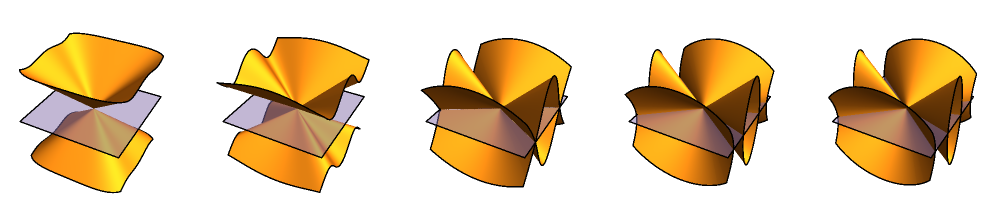}
		\caption{$M(F)$ for $(x y, y z (a x^2+y^2+z^2))$ for $a=3, 6, 9, 12, 15$ }
		\label{aequiv}
	\end{figure}

	\begin{proposition}\label{prop2}
		Let $G_1,G_2:(\mathbb{R}^M, 0) \to (\mathbb{R}^N,0)$ be analytic map germs that are $C^\infty - \mathcal{A}-$equivalent.  If $G_1$ is tame with respect to a proper function $\rho_1$, then there exists another proper function $\rho_2$ such that $G_2$ is tame with respect to $\rho_2$.
	\end{proposition}
	
	\begin{proof}
		In light of the above corollary, we only need to prove this for right-equivalence.  So, assume that there exists a diffeomorphism germ $F:(\mathbb{R}^M, 0) \to (\mathbb{R}^M, 0)$ with $G_1 \circ F = G_2$.  Define $\rho_2 = \rho_1 \circ F$.  Since $\rho_1$ is a proper function and $F$ is a diffeomorphism that maps the origin to itself, $\rho_2$ is also a proper function.  
		
		Since $\d G_2(x)  = \d G_1(F(x)) \cdot \d F(x)$ and $\d F(x)$ is nonsingular, it follows that $\Sing G_2 = F^{-1}(\Sing G_1)$.  Also, since $\rho_2 = \rho_1 \circ F$, we have $\d \rho_2(x) = \d \rho_1(F(x)) \cdot \d F(x)$.  It follows that $M(G_2) = \Sing (G_2, \rho_2) = F^{-1}(\Sing (G_1, \rho_1)) = F^{-1}(M(G_1))$, where $M(G_2)$ is calculated with respect to $\rho_2$ and $M(G_1)$ is calculated with respect to $\rho_1$.  Consequently, $$\overline{M(G_2) \m \Sing G_2} \cap \Sing G_2 = F^{-1}(\overline{M(G_1) \m \Sing G_1} \cap \Sing G_1)$$
		and so the tameness of $G_1$ with respect to $\rho_1$ implies the tameness of $G_2$ with respect to $\rho_2$.
	\end{proof}
	
	To recap: Tameness with respect to a particular proper function $\rho$ is left-invariant (but not right or $\mathcal{A}$-invariant), while being tame with respect to some proper function is $\mathcal{A}$-invariant.

	\begin{corollary} 
		Let $G:(\mathbb{R}^M, 0) \to (\mathbb{R}^N,0)$ be analytic map germ  and $F:(\mathbb{R}^M, 0) \to (\mathbb{R}^M,0)$ a smooth diffeomorphism map germ. If $G$ is tame with respect to a proper function $\rho$, then the pullback map germ $F^{*}G$ is tame with respect to the pullback function $F^{*}\rho .$
	\end{corollary}
	
	\begin{proof}
		This follows from the proof of Proposition \ref{prop2}.   
	\end{proof}

	\section{Tameness of the composite singularities}
	Consider the following diagram of analytic map germs 
	
	$$\begin{tikzcd}
	(\bR^{M},0)\rar{F}\arrow[bend right]{rr}[black,swap]{H=G\circ F}  & (\bR^{N},0) \rar{G}  & (\bR^{K},0)
	\end{tikzcd}$$
	
	We have seen that if $\Disc F =\{0\}$ and $\Disc G =\{0\}$ then $\Disc H =\{0\}.$ To obtain a Milnor tube fibration for $H$ induced from the composite map germs $F$ and $G$, we first ask the following natural question:
	
	\vspace{0.2cm}
	
	\begin{question}\label{q1}
		If the map germs $F:(\mathbb{R}^M, 0) \to (\mathbb{R}^N,0)$ and  $G:(\mathbb{R}^N, 0) \to (\mathbb{R}^K,0)$,  $ M\ge N \ge K \ge 2,$ are both tame, is it true that the composition map germ $H=G\circ F$ is also tame?	
	\end{question}
	
	This is a very natural question to ask, and we should point out that it is motivated by some particular cases where the answer is known to be yes.  In addition to some cases listed in the Introduction, we briefly review some of the known cases below.
	
	\vspace{0.2cm}
	
	If $G:(\mathbb{R}^N,0) \to (\mathbb{R}^K,0), G(y_{1},\ldots, y_{N})=(y_{1},\ldots, y_{K})$ is the canonical projection and $F$ is tame, then $H$ must be tame.  To our knowledge, for $F$ with an isolated singular point at the origin, it was first approached by Y. Iomdin in \cite{Io}. It was also approached in \cite{ADD} by using completely independent tools and techniques, where they also proved formulae for the Euler characteristics of the Milnor fibers. Later, for the general case of $F$ tame with nonisolated singularities ($\Disc F =\{0\}$) it was approached by N. Dutertre and R. Ara\'ujo dos Santos in \cite[Theorem 6.3, p. 4860]{DA}, as a consequence of answering a conjecture stated by Milnor in \cite[p.100]{Mi} which claimed that for such a $F$ and $G$ the fiber of $H=G\circ F$ in the Milnor tube fibration is homeomorphic to the Milnor fiber of $F$ times the $(N-K)$-dimensional closed cube $[0,1]^{N-K}$. 
	
	\vspace{0.2cm}
	Recently in \cite{CT23} the authors considered the very next case $G$ with $\Sing G \subseteq \{0\}$ and proved that the answer to the Question \eqref{q1} is also yes.
	
	\vspace{0.2cm}
	
	Unfortunately, for $G$ with $\dim \Sing G >0$ the answer in general is no, as can be seen below in the Example \ref{nottame}. However, with the condition of $F$ being tame, we introduce a special condition, Condition \eqref{tame} in Theorem \ref{eqHtame}, which provides a characterization of the tameness of $H$. 
	
	\vspace{0.2cm}
	
	We should point out that any map germ $G$ with $\Sing G \subseteq \{0\}$ naturally satisfies our Condition \eqref{tame}. Therefore, that is the reason\footnote{Of course, under the general hypothesis of $F$ being tame with $\Disc F =\{0\}.$} of the tameness of $H=G\circ F$ in the following cases: G is the canonical projection (or any submersion in general) or $G$ has an isolated singular point at the origin. 
	
	\vspace{0.3cm}
	
	The next example provides a negative answer for Question \eqref{q1} in general.
	
	\vspace{0.2cm}
	
	\begin{example}\label{nottame} Let $F:\bR^{4}\rightarrow \bR^{3}$ and $G:\bR^{3}\rightarrow \bR^{2}$ two real analytic maps germ giving by $F(x,y,z,w):=(x,y,z(x^{2}+y^{2}+z^{4}))$ and $G(u,v,t):=(uv,vt)$. Consequently, $H(x,y,z,w)=(xy,yz(x^2+y^2+z^4))$.  By straightforward calculation, one has that $F,G$ are tame, $\Disc F=\{0\}$, $\Disc G=\{0\}$ and 
		$M(H)=\{y=0\}\cup \{w=0,p(x,y,z)=0\}$, where $p(x,y,z)=x^{4}+5 x^{2} z^{4}-x^{2} z^{2}-y^{4}-5 y^{2} z^{4}+3 y^{2} z^{2}+z^{6}$. Now, setting $b(y,z):=z^{2}(5z^{2}-1)$ and $c(y,z):=-y^4+(3z^{2}-5z^{4})y^{2}+z^{6}$, then 
		
		\begin{center}
			$p(x,y,z)=x^{4}+b(y,z)x^{2}+c(y,z)=0.$
		\end{center}
		
		By continuity, it is possible to choose an $\epsilon>0$ such that for all $y,z\in \bR$ with $\epsilon>|y|,|z|>0$ one has that $b(y,0)^{2}-4c(y,0)\geq 0$ and consequently, the branch 
		
		\begin{center}
			$x(y,z)=\sqrt{\frac{(1-5z^{2})z^{2}+\sqrt{b(y,z)^{2}-4c(y,z)}}{2}}$
		\end{center}
		
		is well defined for all $y,z\in \bR$ with $\epsilon>|y|,|z|>0$.
		
		By construction, the branch $\phi(y,z):=(x(y,z),y,z,0)\subset M(H)\setminus \Sing H$ for all $y,z\in \bR$ with $\epsilon>|y|,|z|>0$. Consequently $\phi(0,z)\in \overline{ M(H)\setminus \Sing H}\cap\Sing H$, for all $z\in \bR$ with $|z|<\epsilon$ and $H$ is not tame. See Figure \ref{pic-ivan1} for an illustration.
		
		\begin{figure}	
			\includegraphics[width=.4\textwidth]{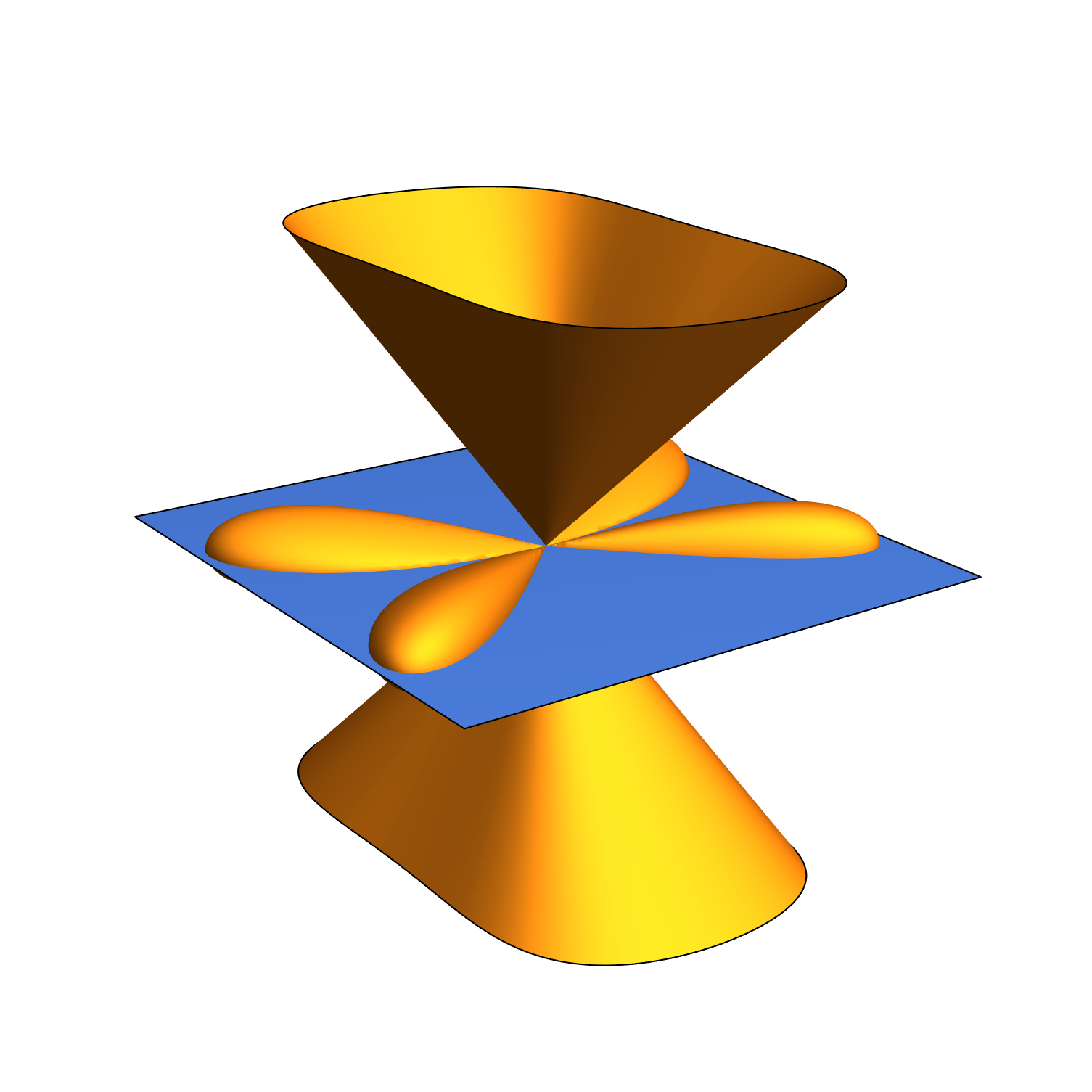}
			\caption{$M(H)$ (yellow and blue) and $\Sing H$ (blue) for Example~\ref{nottame}}
			\label{pic-ivan1}
		\end{figure}
	\end{example} 
	
	\subsection{Main results and its consequences}
	
	Before we prove the main Theorem of this section let us recall some basic facts on the properness of a map and also fix some notations.
	
	\begin{remark}\label{clo} Let $F,G$ and $H=G\circ F$ be analytic map germs as before. For all $\epsilon >0$ small enough, denote by $M(H):=M(H)\cap \bar{B}_{\epsilon}$, similarly, $M(F):=M(F)\cap \bar{B}_{\epsilon}$, where all sets are considered as a germ of sets once they are well defined. An easy exercise shows that 
		\begin{equation}\label{clo1}
		F(\overline{M(H)\setminus \Sing H})=\overline{F(M(H)\setminus \Sing H)}= \overline{F(M(H))\setminus \Sing G}. 
		\end{equation}
	\end{remark}
	
	We now state our main theorem.
	
	\begin{theorem}\label{eqHtame}
		For $ M\ge N \ge K \ge 2,$ $F:(\mathbb{R}^M, 0) \to (\mathbb{R}^N,0)$ tame with $\Disc F=\{0\}$ and $G:(\mathbb{R}^N, 0) \to (\mathbb{R}^K,0)$ with $\{0\} \subseteq \Sing G,$ then the composition $H=G\circ F$ is tame if and only if the following condition holds true:
		\begin{equation}\label{tame}
		\overline{F(M(H)\setminus \Sing H)}\cap \Sing G\subseteq\{0\}.
		\end{equation}
		
	\end{theorem}
	
	\begin{proof}
		To prove the ``if'' implication, take a sequence $\{y_{n}\}\subset \overline{F(M(H)\setminus \Sing H)}\cap \Sing G$ such that $y_{n}\to 0.$ 
		
		\vspace{0.2cm}
		
		From Remark \ref{clo}, the first equality of \eqref{clo1}, it follows that there exists a sequence $\{x_{n}\}\subset \overline{M(H)\setminus \Sing H}$ such that $y_{n}=F(x_{n}).$ 
		
		\vspace{0.2cm}
		
		Since ${y_{n}}\in \Sing G$ also, one has $x_{n}\in \overline{M(H)\setminus \Sing H} \cap \Sing H.$ By the tameness of $H$ there must be a natural number $n_{0}>0$ great enough such that for all $n\geq n_{0}$ one has $x_{n}=0.$ Hence, $y_{n}=F(0)=0$ for all $n$ great enough and the ``if'' condition proved.
		
		\vspace{0.2cm}
		
		For the converse implication, let us now take a sequence $\{x_{n}\}\subset \overline{M(H)\setminus \Sing H} \cap \Sing H$, such that $ x_{n}\to 0$.
		
		\vspace{0.2cm}
		
		Thus $y_{n}=F(x_{n})\in F(\overline{M(H)\setminus \Sing H})\cap \Sing G.$ Applying Condition \eqref{tame} there must be a great enough $n_{0}>0$  such that for all $n\geq n_{0}$, $y_{n}=0$ and thus $x_{n}\in F^{-1}(0)$.
		
		\vspace{0.2cm}
		
		Hence we conclude that $x_{n}\in \overline{M(H)\setminus \Sing H} \cap V_{F}\subseteq \overline{M(F)\setminus V_{F}}\cap V_{F}\subseteq \{0\},$ where the first inclusion follows by Lemma \ref{l1} and the second inclusion follows by the tameness of $F.$ Therefore, for all $n$ great enough the sequence $x_{n}=0$ and the proof is finished.
	\end{proof}

	\begin{remark}
		In light of Theorem \eqref{eqHtame}, we can graph $\overline{F(M(H)\m \Sing H)}$ and $\Sing G$ for Example~\ref{nottame} (see Figure~\ref{pic-ivan2}). Doing so, we see that the intersection of the sets is not contained in $\{0\}$, showing that Example~\ref{nottame} fails Condition~\ref{tame}.  
	\end{remark}
	
	\begin{figure}	
		\includegraphics[width=.4\textwidth]{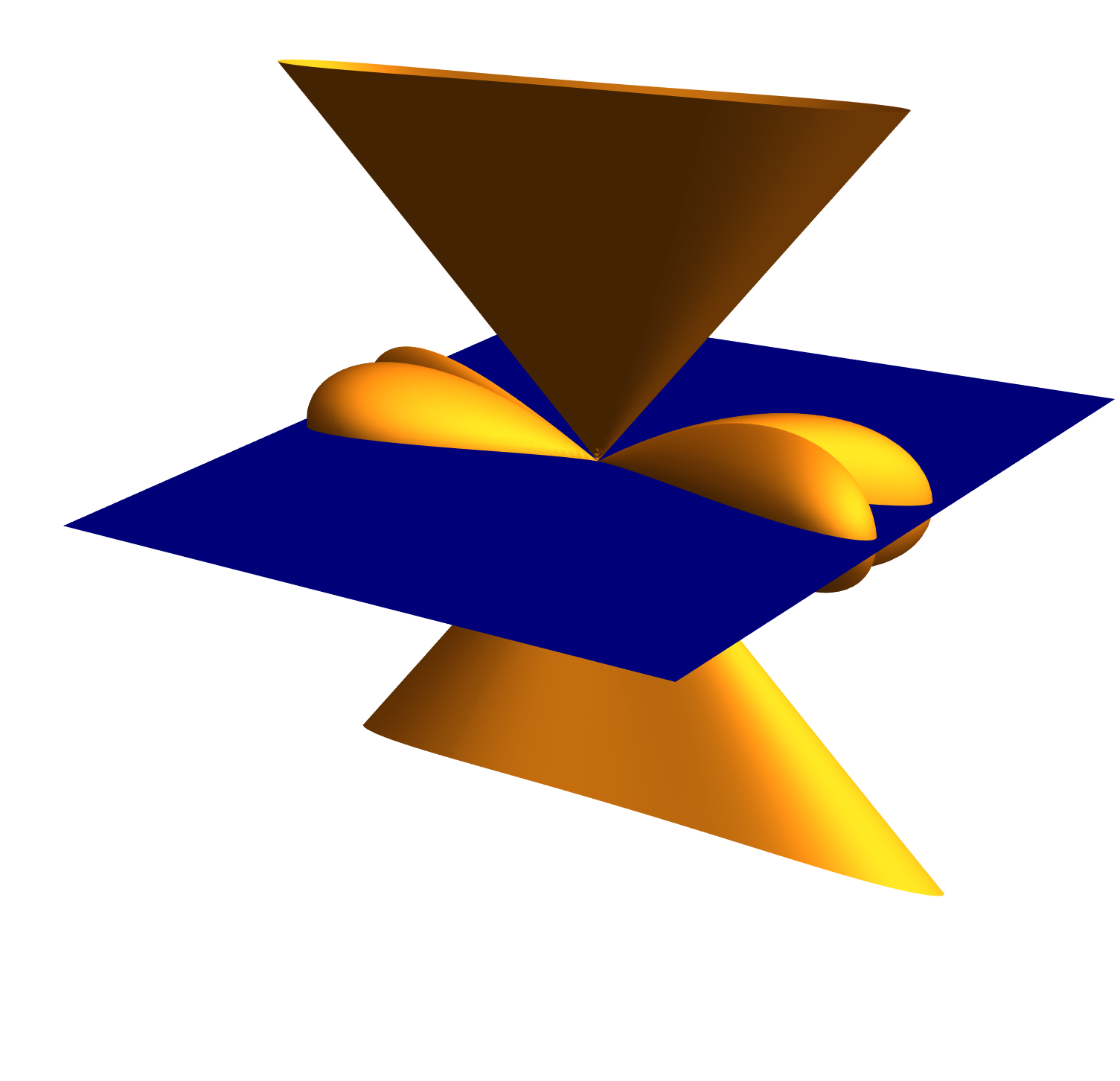}
		\caption{$\overline{F(M(H)\setminus \Sing H)}$ (yellow) and $\Sing G$ (blue) for Example~\ref{nottame}}
		\label{pic-ivan2}
	\end{figure}

	\begin{corollary}[\cite{DA}, Lemma 4.1]
		Let $F:(\mathbb{R}^M, 0) \to (\mathbb{R}^N,0)$, $ M\ge N \ge 3 $, be a  non-constant analytic map germs such that $F$ is tame and $ \Disc F = \{0\} $.  If $\pi:(\mathbb{R}^N, 0) \to (\mathbb{R}^{N-1},0)$ is the canonical projection, then the composition $ \pi \circ F $ is tame and $\Disc (\pi \circ F) = \{0\}$. Hence, there exist Milnor tube fibrations for both $F$ and $\pi \circ F$.
	\end{corollary}
	
	\begin{proof} The proof follows easily as a consequence of Theorem \eqref{eqHtame} because $F$ is tame with $\Disc F =\{0\}.$ Moreover, since $\pi$ is a submersion,  $\Sing \pi =\emptyset$ and Condition \eqref{tame} is trivially satisfied. Therefore, the composition $\pi \circ F$ is tame.
	\end{proof}
	
	\vspace{0.2cm}
	
	\begin{corollary}[\cite{CT23}, Theorem 3.2]
		Let $F:(\mathbb{R}^M, 0) \to (\mathbb{R}^N,0)$ and $G:(\mathbb{R}^N, 0) \to (\mathbb{R}^K,0), M\ge N \ge K \ge 2$ be analytic map germs such that F is tame and $\Disc F=\{0\}$, and that G has an isolated singular point at the origin. Then $H = G\circ F$ is tame and has a local tube fibration.
	\end{corollary}
	
	\begin{proof}
		Since $\Sing G\subseteq \{0\}$, Condition \eqref{eqHtame} is trivially satisfied.
	\end{proof}
	
	\vspace{0.2cm}
	
	Next, we provide an example satisfying the hypothesis of Theorem \eqref{eqHtame} but for a $G$ with $\dim \Sing G >0.$ 
	
	\begin{example}
		Consider the map germs $F:\mathbb{R}^{4} \rightarrow \mathbb{R}^{3}$ and $G:\mathbb{R}^{3}\rightarrow \mathbb{R}^{2}$ given by $F(x,y,z,w)=(x,y,z(x^{2}+y^{2}+z^{2}+w^{2}))$ and $G(u,v,t)=(ut,vt)$. Consequently, we have $H(x,y,z,w)=(xz(x^{2}+y^{2}+z^{2}+w^{2}),yz(x^{2}+y^{2}+z^{2}+w^{2}))$. One can easily show that $F$, $G$ are tame, $\Disc F=\{0\},\Disc G=\{0\},\Sing G=\{t=0\}$,  $\Sing H =\{z=0\}$, and  $M(H)=\{z=0\}\cup\{w=0,z^{2}=x^{2}+y^{2}\}$. Furthermore, one can show $\overline{F(M(H)\m \Sing H)}=\{t^2 = 4(u^2+v^2)^3\}$ Therefore, we have that 
		
		\begin{center}
			$\overline{F(M(H)\m \Sing H)}\cap \Sing G= \{0\}.$
		\end{center}     
		
		By Theorem \ref{eqHtame}, $H$ is tame.
		
	\end{example}
	
	The next example shows that the assumption that $F$ is tame is necessary; Condition \eqref{tame} is sharp in the sense that there are counterexamples if $F$ is not tame.

	\begin{example}\label{fnotame}
		Let $F:(\bR^{4},0) \to (\bR^{3},0), F(x,y,z,w)=(x^2+y^2,z(x^2+y^2),w(x^2+y^2))$ and $G:(\bR^{3},0) \to (\bR^{2},0), G(u,v,t)=(u(u^2+v^2+t^2),v).$ It follows that $\Sing F \subseteq V_{F}=\{x=y=0\}$ but $F$ is not tame since $F^{-1}(-\delta,0,0)$ is empty, for any $\delta>0.$ On the other hand, $\Sing G =\{0\}$ and Condition \eqref{tame} is satisfied. However, by the argument for $F$, the composition $H=G\circ F$ is not tame.
		
	\end{example}
	
	The next example shows that the composite map germ $H=G\circ F$ can be tame even if the map germ $G$ is not tame.
	
	\begin{example} \label{exdan} Let $F:(\bR^{4},0) \to (\bR^{3},0), F(x,y,z,w)=(\dfrac{1}{3}xw,yw,zw)$, $G:(\bR^{3},0) \to (\bR^{2},0), G(u,v,t)=(ut,vt(9u^2+v^2+t^2)$ and $H(x,y,z,w)= G \circ F = (\dfrac{1}{3}w^2xz,w^4yz(x^2+y^2+z^2))$.  It is straightforward to show that $F$ is tame and $\Disc F = \{0\}$, while $\Sing G = \{t=0\}$ and $M(G) = \{t=0\} \cup \{t^4 + 8 t^2 u^2 - 9 u^4 + 6 u^2 v^2 - v^4\}$. Thus
		$$
		\overline{M(G)\setminus \Sing G} \cap \Sing G = \{t=0, v^2 = 3 u^2\}
		$$
		and hence $G$ is not tame.  It is a more difficult computation, but it can be shown that $\Sing H = \{w=0\} \cup \{z=0\}$ and $M(H) = \{w=0\} \cup \{z=0\} \cup \{z^2 = x^2 + y^2, w^2 = 2 z^2\}$.  From this it is clear that $\overline{M(H) \setminus \Sing H} \cap \Sing H = \{0\}$, and therefore $H$ is tame.  
		
	\end{example}
	
	The next example shows that the composite map germ $H=G\circ F$ can be tame even if the map germ $F$ is not tame.
	
	\begin{example} \label{exdan2} Let $F:(\bR^{5},0) \to (\bR^{4},0), F(x,y,z,w, k)=(x, y, z, x w)$, $G:(\bR^{4},0) \to (\bR^{2},0), G(u,v,t, s)=(ut, vt)$ and $H(x,y,z,w, k)= G \circ F = (xz, yz)$.  It is easy to show that $F$ is not tame, $G$ is tame, and $H$ is tame.   
		
	\end{example}
	
	\begin{remark}
		We now have examples (\ref{nottame}, \ref{exdan}, and \ref{exdan2}) showing that it is possible for two of the functions $F$, $G$, and $H$ to be tame, but the third is not.  By combining Examples \ref{exdan} and \ref{exdan2}, we can find two functions that are not tame but their composition is tame.  For example, neither $F(x, y,z w) = (\frac{1}{3} x w, y w, z w, w)$ nor $G(u, v, t, s) = (u t, v t (9 u^2 + v^2 + t^2)$ are tame, but their composition $H = G \circ F$ is tame.      
	\end{remark}
	
	\begin{remark}
		\label{rem:nottame}
		As mentioned above, when we say that a function is tame, we mean that the function is tame with respect to $\rho_{E}(x)=\|x\|^{2}$.  In particular, the non-tame functions in Examples \ref{nottame} and \ref{exdan} are both tame with respect to a different choice of $\rho$ ($\rho = x^2+y^2+z^4+w^2$ for Example \ref{nottame} and $\rho = 9u^2+v^2+t^2$ for Example \ref{exdan}).   
	\end{remark}
	
	To finish this section, we introduce a sufficient condition for the tameness of the composition $H=G\circ F$ under the hypothesis that $F$ and $G$ are both tame.
	
	\begin{proposition} \label{incl}
		Let $F:(\mathbb{R}^M, 0) \to (\mathbb{R}^N,0)$ and $G:(\mathbb{R}^N, 0) \to (\mathbb{R}^K,0), M\ge N \ge K \ge 2$, be analytic map germs, with $F$ tame and $\Disc F=\{0\}$. Suppose further that $F(M(H))\subseteq M(G)$ as a germ of sets. Then if $G$ is tame, so is the composition $H=G\circ F$.    
	\end{proposition}
	
	\begin{proof}
		From the inclusion $F(M(H))\subseteq M(G)$, it follows that $\overline{F(M)\m \Sing G} \subseteq \overline{M(G) \m \Sing G}$.  Now applying Remark \ref{clo}, one can see that the following inclusions hold true:  
		
		\vspace{0.2cm}
		
		\begin{eqnarray*}    
			\overline{F(M(H)\setminus \Sing H)}\cap \Sing G & \subseteq & \overline{F(M(H))\setminus \Sing G}\cap \Sing G \\ &\subseteq & \overline{M(G) \m \Sing G} \cap \Sing G \\ & \subseteq& \{0\}
		\end{eqnarray*}
		where in the last inclusion used the tameness of $G$. 
		
		\vspace{0.2cm}
		
		Therefore, the inclusion $\overline{F(M(H)\setminus \Sing H)}\cap \Sing G \subseteq \{0\}$ holds and by Theorem~\ref{eqHtame}, $H$ is tame. 
	\end{proof}
	
	\vspace{0.2cm}

	We can extend Proposition \ref{incl} by replacing the set inclusion with an equality.
	
	\begin{corollary}\label{rec}
		Let $F:(\mathbb{R}^M, 0) \to (\mathbb{R}^N,0)$ and $G:(\mathbb{R}^N, 0) \to (\mathbb{R}^K,0), M\ge N \ge K \ge 2$, be analytic map germs, with $F$ tame and $\Disc F=\{0\}.$ Suppose that $F(M(H))=M(G)$ as a germ of set at the origin. Then $G$ is tame if and only if the composition $H=G\circ F$ is tame.    
	\end{corollary}
	
	\begin{proof}
		In light of Proposition \ref{incl} it suffices to prove that if
		$M(G)\subseteq F(M(H))$, then $H$ tame implies $G$ tame. By our assumptions we have $$\overline{M(G) \m \Sing G}\cap \Sing G \subseteq \overline{F(M(H))\setminus \Sing G}\cap \Sing G \subseteq \{0\}$$
		where the last inclusion follows from the tameness of $H$. 
	\end{proof}
	
	\begin{remark} A summary of all relationship between tameness of $F, G$ and the composition $H$ is as follows:
		
		\begin{itemize}
			\item [(1)] Example \eqref{nottame} shows that $F$ and $G$ tame $\not\Rightarrow$ $H$ tame. Proposition \ref{incl} provides a sufficient condition for the implication to be true;
			
			\item [(2)] Example \eqref{exdan} shows that $F$ and $H$ tame $\not\Rightarrow $ $G$ tame. Corollary \ref{rec} provides a sufficient condition for the implication to be true;
			
			\item [(3)] Example \eqref{exdan2} shows that $G$ and $H$ tame $\not\Rightarrow $ $F$ tame. 
		\end{itemize}    
	\end{remark}
	
	%\begin{example} \label{newex}
	%Let $F(x,y,z,w)=(x^2+y^2,z(x^2+y^2),w(x^2+y^2)),$ $G(u,v,t)=(v,t).$ Hence, %$H(x,y,z,w)=(z(x^2+y^2),w(x^2+y^2)).$ It is easy to see that $V_{F}=\{x=y=0\},$ %$\Disc F =\{0\},$ but $M(F)=\mathbb{R}^4.$ The map germ $G$ is a submersion and %$V_{H}=\{x=y=0\}\cup \{z=w=0\},$ $\Sing H =\{x=y=0\}.$ Moreover, $M(H)=\{x=y=0\}\cup \%{x^2+y^2-2z^2-2w^2=0\},$ $\overline{M(H)\m \Sing H}=\{x^2+y^2-2z^2-2w^2=0\}$ and hence $H$ is tame. 
	%\end{example}

	\section{Applications}
	
	For $M\geq N \geq K \geq 2$, consider the following diagrams: 
	
	$$\begin{tikzcd}
	(\bR^{M},0)\rar{F}\arrow[bend right]{rr}[black,swap]{H=G\circ F}  & (\bR^{N},0) \rar{G}  & (\bR^{K},0)
	\end{tikzcd}$$
	
	and
	
	\begin{center}
		\begin{equation}\label{d2}
		\begin{tikzcd}[column sep=small]
		\bar{B}^M_{\epsilon}\cap H^{-1}(\bar{B}_{\eta}^{K}-\{0\}) \arrow{r}{H_{|}}  \arrow{rd}{\rho_{K}\circ H} 
		& \bar{B}_{\eta}^{K} - \{0\} \arrow{d}{\rho_{K}(z):=\|z\|^{2}} \\
		& (0,\eta^{2}]
		\end{tikzcd}
		\end{equation}
	\end{center}
	
	\vspace{0.2cm}
	
	%\begin{tikzcd}[column sep=small]
	%A \arrow{r}{i}  \arrow{rd}{g} 
	%  & A3 \arrow{d}{h} \\
	%    & A2
	%\end{tikzcd}
	
	We may consider $\eta>0$ small enough such that $\rho_{K}\circ H$ is a proper, smooth submersion function, thus it is a smooth trivial fibration. Therefore, the following diffeomorphism holds true:

	\begin{equation}
	\bar{B}^M_{\epsilon}\cap H^{-1}(\bar{B}_{\eta}^{K}-\{0\}) \cong \bar{B}^M_{\epsilon}\cap H^{-1}(S_{\eta}^{K-1}) \times (0,\eta^{2}]
	\end{equation}
	
	\vspace{0.2cm}
	
	On the other hand, the following commutative diagram 
	
	\vspace{0.2cm}
	
	\begin{equation}\label{d3}
	\begin{tikzcd}
	\bar{B}^M_{\epsilon}\cap H^{-1}(\bar{B}_{\eta}^{K}-\{0\}) \arrow[r, "\cong"] \arrow[d, "H_{|}"]
	& \bar{B}^M_{\epsilon}\cap H^{-1}(S_{\eta}^{K-1}) \times (0,\eta^{2}]
	\arrow[d, "\Psi"] \\
	\bar{B}_{\eta}^{K}-\{0\} \arrow[r, "\cong" blue]
	&  S_{\eta}^{K-1}\times (0,\eta^{2}] 
	\end{tikzcd}
	\end{equation}
	where $\Psi=(H_{|},\rm{Id})$, says that the study of the Milnor tube fibration \eqref{tmil} may be reduced, up to homotopy, to the Milnor fibration of Corollary \ref{ttf2}.
	
	\vspace{0.2cm}
	
	Now consider $F$ and $G$ with $\Disc F=\{0\}$ and $\Disc G=\{0\}$ then $\Disc H=\{0\}$. Suppose further that $F$, $G$, and $H$ are tame. Then by Thereom \ref{ttf}, we can find $\epsilon_{0}>0$ small enough such that for all $0<\epsilon \leq \epsilon_{0}$ there exist $\eta$ and $\tau$ with $0<\eta \ll \tau \ll \epsilon $ for which the three Milnor tubes exist:
	
	$$H_{|}:\bar{B}^M_{\epsilon}\cap H^{-1}(\bar{B}_{\eta}^{K}-\{0\})\to \bar{B}_{\eta}^{K}-\{0\}$$

	$$G_{|}:\bar{B}^N_{\tau}\cap G^{-1}(\bar{B}_{\eta}^{K}-\{0\})\to \bar{B}_{\eta}^{K}-\{0\}$$

	$$F_{|}:\bar{B}^M_{\epsilon}\cap F^{-1}(\bar{B}_{\tau}^{N}-\{0\})\to \bar{B}_{\tau}^{N}-\{0\}$$
	
	Denote by $F_{F}, F_{G}$ and $F_{H}$ the respective Milnor fibers induced by fibration projections $F_{|},G_{|}$ and $H_{|}.$
	
	\vspace{0.2cm}
	
	One can choose any $z \in \bar{B}_{\eta}^{N}-\{0\}$ and then $F_{H}=F^{-1}(G^{-1}(z))=F^{-1}(F_{G}).$ Thus the Milnor tube fibration  $$F_{|}:\bar{B}^M_{\epsilon}\cap F^{-1}(\bar{B}_{\eta}^{N}-\{0\})\to \bar{B}_{\eta}^{N}-\{0\},$$ restricts to the sequence of map   
	%$$F_{|}: F_{H} \twoheadrightarrow F_{G}$$
	
	\begin{equation}\label{dia1}
	F_F \hookrightarrow F_H  \xlongrightarrow{\text{$F_{|}$}} F_G.
	\end{equation}
	
	which is also a smooth locally trivial fibration with fiber $F_{F}.$
	
	\vspace{0.3cm}
	
	Therefore, if $ F_G $ is contractible, then one gets the following diffeomorphism of the Milnor's fibers $ F_H \cong F_F \times F_G.$ 
	
	\vspace{0.2cm}
	
	In particular, it provides an easy proof for a conjecture stated by J. Milnor in \cite{Mi}, p. 100, which we remind of below for completeness. We should point out that Milnor's conjecture was stated for $F$ with $\Sing F=\{0\}$, but below we will use the more general case of $\Disc F=\{0\}$, as stated and proved in \cite{DA}, using completely different tools. { See also \cite[Corollary 3.3, p. 6]{CT23} for a recent proof.}
	
	\begin{corollary}[Milnor's conjecture, \cite{Mi}{~p. 100}]
		Let $F:(\mathbb{R}^M, 0) \to (\mathbb{R}^N,0)$, $ M\ge N \ge 3 $, be analytic map germs such that $F$ is tame and $\Disc F = \{0\}$. If $\pi:(\mathbb{R}^N, 0) \to (\mathbb{R}^{N-1},0)$ is the canonical projection, then the fiber $ F_{\pi \circ F}$ is homeomorphic to $F_{F}\times [-1,1].$ 
	\end{corollary}
	
	\begin{remark}
		
		\begin{enumerate}
			
			\item[(1)] As a direct consequence of the fibration \eqref{dia1}, one gets a multiplicative formula for the Euler characteristic of the fibers $\chi(F_{H})=\chi(F_{F})\chi(F_{G})$. A proof for this formula using spectral sequences may be found in the book of E. H. Spanier \cite[Section 3, p. 481]{Sp}, or the reader may consult the reference \cite{KMR} for another proof. 
			
			\vspace{0.2cm}
			
			\item[(2)] Since $H:(\bR^{M},0)\to (\bR^{K},0)$ admits a Milnor tube fibration 
			
			$$H_{|}:\bar{B}^M_{\epsilon}\cap H^{-1}(\bar{B}_{\eta}^{K}-\{0\})\to \bar{B}_{\eta}^{K}-\{0\},$$ 
			it follows from diagram \eqref{d3} that up to homotopy, one can restrict the map to the fibration  
			$$F_H \hookrightarrow \bar{B}^M_{\epsilon}\cap H^{-1}(S_{\eta}^{K-1})  \xlongrightarrow{\text{$H_{|}$}} S_{\eta}^{K-1}.$$
			
			\vspace{0.2cm}
			
			Therefore, $\chi(\bar{B}^M_{\epsilon}\cap H^{-1}(S_{\eta}^{K-1}))=\chi(S_{\eta}^{K-1})\chi(F_{H}).$
			
			\vspace{0.2cm}
			
			\item[(3)] Denote by $T_{H}:=\bar{B}^M_{\epsilon}\cap H^{-1}(S_{\eta}^{K-1})$ the Milnor tube of $H$, and by $T_{G}$ and $T_{F}$ the Milnor tubes of $G$ and $F$, respectively. Under the condition that $F,G$ and $H$ admit Milnor tube fibrations, it follows from the items $(1)$ and $(2)$ above that
			
			$$\chi(T_{H})=\chi(S_{\eta}^{K-1})\chi(F_{H})=\underbrace{\chi(S_{\eta}^{K-1})\chi(F_{G})}_{=\chi(T_{G})}\chi(F_{F})=\chi(T_{G})\chi(F_{F}).$$

		\end{enumerate}

	\end{remark}
	
	\subsection{The isolated singular points cases} 
	
	In \cite{ADD} the authors proved a special formula for the Euler characteristic of the Milnor fiber for an isolated singularity map germ as follows.
	
	\begin{theorem} \cite{ADD}
		Let $F: (\bR^{M},0) \to (\bR^N,0), M > N \geq 2$, be a polynomial map germ with an isolated singularity at the origin, and consider
		$F(x)=(F_{1}(x),F_{2}(x),\ldots, F_{N}(x))$,
		an arbitrary representative of the germ. Denote by $\deg_{0}(\nabla F_{i}(x))$, for $i=1, \ldots, N$ the topological degree of the map $\epsilon \dfrac{\nabla F_{i}}{\|\nabla F_{i}\|}: S_{\epsilon}^{M-1}\to S_{\epsilon}^{M-1}$, for $\epsilon >0$ small enough.
		
		\vspace{0.2cm}
		
		\begin{enumerate}
			\item [(i)] If $M$ is even, then $\chi (F_{F}) = 1- \deg_{0} \nabla F_{1}.$ Moreover, we have  $$\deg_{0} \nabla F_{1}=\deg_{0} \nabla F_{2}=\cdots = \deg_{0} \nabla F_{N}.$$
			
			\item [(ii)] If $M$ is odd, then $\chi(F_{F}) = 1.$ Moreover, we have $\deg_{0} \nabla F_{i}=0$ for $i=1,2,\ldots, N.$
		\end{enumerate}
	\end{theorem}
	
	We can write the two formulas as $\chi(F_{F})=1-\dfrac{1}{2}\chi(S^{M})\deg_{0} \nabla F_{1},$ where $S^{M}$ is the unit $M$-dimensional sphere.
	
	\vspace{0.2cm}
	
	If we consider the composition map germ $H=G\circ F$ in the special cases of $\Sing F=\{0\}$ and $\Sing G=\{0\}$, then $\Disc H=\{0\}$ and $F$, $G$ and $H$ are tame. Hence, the Euler characteristic of $F_{H}$ admits a topological degree formula as follows. 
	
	\begin{corollary} Let $F: (\bR^{M},0) \to (\bR^N,0)$ and
		$G: (\bR^{N},0) \to (\bR^K,0), M > N > K \geq 2$, be polynomials map germs, with $\Sing F=\{0\}$ and $\Sing G=\{0\}.$ Denote its representatives by $F(x)=(F_{1}(x),F_{2}(x),\ldots, F_{N}(x)),$ and $G(x)=(G_{1}(x),G_{2}(x),\ldots, G_{K}(x)),$ and consider the composition map germ $H=G\circ F.$ Then, their respective Milnor fibers are related as $$\chi(F_{H})= 1-\dfrac{1}{2}\chi(S^{M})\deg_{0} \nabla F_{1} -\dfrac{1}{2}\chi(S^{N})\deg_{0} \nabla G_{1} +\dfrac{1}{4}\chi(S^{M})\chi(S^{N})\deg_{0} \nabla F_{1}\deg_{0} \nabla G_{1}.$$ 
		
	\end{corollary}
	
	In particular, it says that if $\chi(F_{H})\neq 0$ then neither of the maps  $\epsilon \dfrac{\nabla F_{1}}{\|\nabla F_{1}\|}: S_{\epsilon}^{M-1}\to S_{\epsilon}^{M-1}$ nor $\epsilon \dfrac{\nabla G_{1}}{\|\nabla G_{1}\|}: S_{\epsilon}^{N-1}\to S_{\epsilon}^{N-1}$ can be homotopic to the identity maps.\footnote{It is enough to check that only for $M$ and $N$ even.}
	
	\vspace{0.2cm}
	
	%\subsection{The non-isolated singular case} Given $F:(\mathbb{R}^M, 0) \to (\mathbb{R}^N,0), F=(F_{1},F_{2},\ldots, F_{N}),$ $M\ge N \ge 3,$ consider the index set $A:=\{1,2,\ldots, N\}$ and denote by $I=\{i_{1},i_{2},\ldots,i_{l}\}$ any subset ordered index on $A.$ Denote by $F_{I}:=(F_{i_{1}},F_{i_{2}},\ldots, F_{i_{l}})$ the restriction map germ. For $\epsilon>0$ small enough, denote by $\mathcal{L}_{I}:=F^{-1}_{I}\cap S_{\epsilon}^{M-1}$ the link of the singularity $F_{I}.$ 
	
	%Now given $I$ as above with lenght $|I|=l\geq 3$ denote by $J:=\{i_{1},i_{2},\ldots,i_{l-1}\}.$ 
	%\vspace{0.2cm}
	
	%It follows from an easy application of our previous results that if $F$ is tame with $\Disc F=\{0\}$, then so are $F_{I}$ and $F_{J}$. 
	
	%\vspace{0.2cm}
	
	%\begin{theorem}[\cite{DA}, Proposition 7.1, p. 4861] $\chi(\mathcal{L}_{J})- \chi(\mathcal{L}_{I})=(-1)^{M-l}2\chi(F_{F}).$
	
	%\end{theorem}
	
	\subsection{A factorization problem} Let us consider again the diagram of map germs:
	
	$$\begin{tikzcd}
	(\bR^{M},0)\rar{F}\arrow[bend right]{rr}[black,swap]{H=G\circ F}  & (\bR^{N},0) \rar{G}  & (\bR^{K},0)
	\end{tikzcd}$$
	
	We have seen that given $F$ and $G$ with $\Sing F=\{0\}$ and $\Sing G=\{0\}$, then $H=G\circ F$ has $\Disc H=\{0\}$ and is tame. It motivates the following result. 
	
	\begin{proposition} In the diagram above suppose that $\Disc F=\{0\},$ $\Disc H=\{0\},$ and that $G$ is $\mathcal{C}^{l}-K-$finitely determined for all $l\geq 0$. Then $\Sing G=\{0\}$ and hence $G$ is tame. Moreover, if further $F$ is tame then $H$ is tame as well.
	\end{proposition}
	
	\begin{proof}
		The proof follows easily from the triangular diagram in Section \ref{sec:composite} and the fact that the condition for $G$ to be $\mathcal{C}^{l}-K-$ finitely determined is equivalent to $\Sing G \cap V_{G}=\{0\}$ in a small neighborhood of the origin. These two facts together imply that $\Sing G =\{0\}$, and hence $G$ is tame. If further $F$ is tame, Condition \eqref{tame} is trivially satisfied and so $H$ is tame as well.  
	\end{proof}
	
	This result also motivates us to state this more general (maybe difficult) question:
	
	\begin{question}
		Given $H$ with $\Disc H=\{0\}$ and tame, is it possible to split $H$ as a composition $H=G\circ F$ of map germs $F$ and $G$, with $F$ tame and $\Disc F=\{0\}$? 
	\end{question}
	
	%\section*{Acknowledgements}
	%The authors thanks to the anonymous referee for the comments which helped to improve the paper. 
	
	%The authors thanks...In particular the first author would like to thank.... the second...
	
	%\section*{Acknowledgements}
	%We would like to thank you for following the above instructions. This will hep to speed up the publication process of your paper.

\end{document}